\documentclass[a4paper,reqno,12pt]{article}

\usepackage{amsmath,amsthm}
\usepackage{amsfonts}
\usepackage{amssymb}
\usepackage{xcolor}
\usepackage{mathrsfs}
\usepackage[colorlinks=true,citecolor=black,linkcolor=black,urlcolor=blue]{hyperref}
\usepackage{multirow}
\usepackage{longtable}
\usepackage{mathtools}

\newcommand{\arxiv}[1]{\href{http://arxiv.org/abs/#1}{\texttt{arXiv:#1}}}

\newtheorem{theorem}{Theorem}[section]
\newtheorem{lemma}[theorem]{Lemma}
\newtheorem{definition}[theorem]{Definition}

\newtheorem{corollary}[theorem]{Corollary}

\newcommand{\NN}{\mathbb{N}}

\newcommand{\FF}{\mathbb{F}}
\newcommand{\EE}{\mathbb{E}}
\newcommand{\PP}{\mathbb{P}}
\newcommand{\Tsequence}{t}

\newcommand{\A}{\mathcal{A}}

\newcommand{\cg}[1]{{\color{blue} #1}}

\newcommand{\F}{\mathbb{F}_{q}^{n}}
\newcommand{\xvec}{\boldsymbol{x}}
\newcommand{\zerovec}{\boldsymbol{0}}

\def\nfrac#1#2{{\textstyle\frac{#1}{#2}}}
\def\dfrac#1#2{\lower0.15ex\hbox{\large$\frac{#1}{#2}$}}

\title{Threshold functions for substructures\\ in random subsets of finite vector spaces}

\author{Changhao Chen\thanks{Supported by Australian Research Council Discovery Project 170100786.}\\
\small Department of Mathematics\\[-0.2ex]
\small The Chinese University of Hong Kong\\[-0.2ex]
\small Shatin, N.T., Hong Kong\\
\small \texttt{changhao.chenm@gmail.com}\\
\and
Catherine Greenhill\thanks{Supported by Australian Research Council Discovery Project 140101519.}\\
\small School of Mathematics and Statistics\\[-0.2ex]
\small UNSW Sydney\\[-0.2ex]
\small Sydney NSW 2052, Australia \\[-0.2ex]
\small \texttt{c.greenhill@unsw.edu.au}
}

\date{1 April 2020}

\begin{document}

\maketitle

\begin{abstract}
The study of substructures in random objects has a long history, beginning with
Erd\H{o}s and R{\' e}nyi's work on subgraphs of random graphs.
We study the existence of certain substructures in random subsets of 
vector spaces over finite fields. First we provide a general framework which 
can be applied to establish coarse threshold results and prove a limiting 
Poisson distribution at the threshold scale.
To illustrate our framework we apply our results to $k$-term arithmetic 
progressions, sums, right triangles, parallelograms and affine planes.  
We also find coarse thresholds for the property that a random
subset of a finite vector space is sum-free, or is a Sidon set.
\vspace*{\baselineskip}

\end{abstract}

\section{Introduction}\label{sec:introduction}

Thresholds for the existence of substructures in random objects have
been well studied in many contexts, for example in
random graphs (see~\cite{Alon,Bo,FK}), random subsets of integers
(see~\cite{RR,RSZ}) and even in
random fractal sets~\cite{SS} where the threshold functions 
are taken with respect to various dimensions of fractal geometry.
While extremal problems ask for conditions which guarantee
that a certain substructure is always present, the random version
relaxes this to ask that the substructure is present with
probability which tends to 1.

Let $G(n,p)$ denote the binomial random graph with $n$ vertices.
The threshold for the property that $G(n,p)$ 
contains a given graph $G$ as a subgraph
depends on the density of the densest subgraph of $G$.
Formally, define
\[ m(G) = \max\{ \rho(H) \mid  H\subseteq G\}\]
where $\subseteq$ denotes the subgraph relation and $\rho(H) = |E(H)|/|H|$
is the density of $H$.  If $m(G) = \rho(G)$
then $G$ is \emph{balanced}, and if $\rho(H) < \rho(G)$ for all
proper subgraphs $H$ of $G$ then $G$ is \emph{stricly balanced}.
As part of their foundational work on random graphs,
Erd{\H o}s and R{\' e}nyi~\cite{ERII} found the threshold for almost-sure
containment of $G$ as a subgraph in $G(n,p)$, when $G$ is
balanced. This was extended to general graphs by Bollob{\' a}s~\cite{Bo1981},
who showed that if $m(G)\geq 1$ then
\[ \lim_{n\to\infty} \Pr\big(G\subseteq G(n,p)\big) =
  \begin{cases} 0 & \text{ if $p n^{1/m(G)} \to 0$,} \\
              1 & \text{ if $p n^{1/m(G)} \to\infty$.} \end{cases}
\]
Attention then turned to the distribution of the number of copies of
a given graph in $G(n,p)$.
Independently, Bollob{\' a}s~\cite{Bo1981} and 
Karo{\' n}ski and Ruci{\' n}ski~\cite{KR} showed that if $G$ is strictly balanced
then, at the threshold scale (that is, when $p = cn^{-1/m(G)}$
for some constant $c>0$), the number of copies of $G$ in $G(n,p)$
tends to a Poisson distribution.
Ruci{\' n}ski and Vince proved and the converse is true~\cite{RV}.  
If $G$ is balanced but not
strictly balanced then the limiting distribution is normal,
as proved by Ruci{\' n}ski~\cite{R}.
Analogous results (on the threshold for existence, and limiting Poisson
distrubution) for subgraphs of random regular graphs were proved by Kim, Sudakov
and Vu~\cite{KSV}.

Researchers in arithmetic combinatorics have studied the
existence of substructures in random sets of integers. 
Several important structures, such as arithmetic progressions,
sum-free sets and Sidon sets, can be described in terms of solutions
(or non-solutions) of a certain system of linear equations.
Ru{\' e}, Spiegel and Zumalac{\' a}rregui~\cite{RSZ} gave a threshold
function for the existence of solutions to a given system of linear equations
in random subsets of $[n] = \{1,2,\ldots, n\}$.  In their result, the
integer matrix $M$ representing the system of linear equations must have full
rank, at least one solution with all entries positive, and for each
$i\neq j$ there must be a solution $\xvec=(x_1,\ldots, x_m)$ with
$x_i\neq x_j$.  The threshold depends on a parameter $c(M)$ which is
reminiscent of the parameter $m(G)$ arising in the study of subgraphs in
random graphs~\cite{ERII}. 
In particular, they show that almost all subsets of $[n]$ of size
$\omega(n^{1-2/k})$ contain arithmetic progressions of length~$k$.
Ru{\' e} et al.\ also established a Poisson limiting distribution
for the number of solutions of the linear system
$M\xvec = \zerovec$, at the threshold scale.


A set $S\subseteq \mathbb{N}$ is a \emph{Sidon set} if all sums
of the form $x+y$ with $x\leq y$ and $x,y\in S$ are distinct.
Kohayakawa, Lee, R{\" o}dl and Samotij~\cite{KLRS}
analysed the maximum size of a Sidon set contained in a random set of integers.

Motivated by this literature, we consider threshold functions for the existence 
of certain substructures in random subsets of vector spaces over finite fields.  

\subsection{Notation and our main results}\label{s:statements}

Let $\mathbb{F}_{q}$ denote the finite field with $q$ elements, where $q$ is a 
prime power, and let $\F$ be the $n$-dimensional vector space over this field. 
When $q$ is fixed and $n$ is large, the set $\F$  is known as the 
\emph{finite field model}, and was suggested by Green~\cite{Green} as a
toy model for testing proof techniques in additive number theory.  This gave 
a simplified setting which retained key features of the original problem.
See the survey of Wolf~\cite{Wolf} for more details.

%
Given $0<\delta<1$, let $E$ be a random subset of $\F$ chosen as follows: 
put $x\in E$ with probability $\delta$ and $x\notin E$ with probability $1-\delta$, 
independently for all $x\in \F$. Denote  the resulting probability space by $\Omega(\F,\delta)$. 
The model $\Omega(\F,\delta)$ is closely related to the Gilbert random graph model $G(n,p)$, 
see for example \cite{FK}, and to random subsets of $[N]=\{1,2,\ldots, N\}$.  
When working with finite fields, it is common to use $p$ to denote
a prime number, rather than a probability.  This is why we use $\delta$
to denote the probability parameter in our model.


Standard asymptotic notation is given at the start of Section~\ref{sec:pre}.
All asymptotics are as $q+n\rightarrow\infty$ unless otherwise stated.

Let $P$ be a combinatorial property and $E\in \Omega(\F,\delta)$. 
Say that $P$ is a \emph{monotone increasing} property if 
the property is inherited by supersets, and a \emph{monotone decreasing} property
if it is inherited by subsets.
Then $t(n,q)$ is a \emph{(coarse) threshold function} for the monotone increasing property $P$ if
the following holds:
\begin{itemize}
\item[(i)] if $\delta = o(t(n,q))$ then $\PP(\text{$E$ satisfies property $P$}) \rightarrow 0$, and 
\item[(ii)] if $\delta = \omega (t(n,q))$ then $\PP(\text{$E$ satisfies property $P$} ) \rightarrow 1$.
\end{itemize}
If $P$ is a decreasing properties then a coarse threshold is defined by  applying
the above definition to the negation of $P$. Most properties we consider involve
containment of a given substructure, so they are increasing properties.
Bollob{\' a}s and Thomason~\cite{BT} proved that every non-trivial monotone 
increasing property of sets has a coarse threshold.

\bigskip

Let $\A=\A_{a}$ be a family of subsets of $\F$ such that each element of $\mathcal{A}$ contains
$a$ vectors, where $a$ is a fixed constant.  
Let $X=X_{\A}(E)$ be the random variable which counts the elements of $\A$ in the random set $E$. 
That is,
\begin{equation}\label{eq:x}
X=X_{\A}(E) =\sum_{T\in \A} {\bf 1}_{E}(T)
\end{equation}
where ${\bf 1}_E(T)$ equals 1 if $T\subseteq E$ and equals 0 otherwise.
We write $|S|$ to denote the cardinality of a set $S$. By linearity of expectation we have 
\begin{equation}
\label{eq:lambdaA}
\EE(X)= |\A|\,\delta^{a}=:\lambda_{\A}. 
\end{equation}
Let Po$(\mu)$ denote the Poisson distribution with mean $\mu$, and write 
$X \xrightarrow{\text{d}} \operatorname{Po}(\mu)$ if $X$ tends in distribution 
to Po$(\mu)$.  

Our main contribution is the following general theorem which can be used
to locate a coarse threshold for the property ``contains an element of $\A$'',
and to prove that that the number of patterns in $\A$ contained in $E$
has a Poisson distribution when its expected value tends to a finite limit.

\begin{theorem}\label{thm:general}
Suppose there exist integer constants $b>0$, $c\ge 0$ such that $|\A|=\Theta (q^{bn-c})$. 
Assume that for any set  $S$  of $k$ distinct points in $\mathbb{F}_{q}^{n}$, the number of elements of $\A$ which contain~$S$ is 
\begin{equation}\label{eq:10}
  \begin{cases}
   O(q^{(b-k)n-c}) & \text{if $1\leq k\leq b-1$}, \\
   O(1) &\text{otherwise}.
  \end{cases}
\end{equation}
Then the event ``contains an element of $\A$'' has a coarse threshold function 
\[
t(n, q)=q^{(c-bn)/a}.
\]
\emph{(}Above, asymptotics are as $q+n\to\infty$.\emph{)}
Furthermore, if 
$\EE(X)\rightarrow \lambda$ for some constant $\lambda>0$ 
then $X \xrightarrow{\text{d}} \operatorname{Po}(\lambda)$,
where asymptotics are as
\[ \begin{cases} q+n\to\infty & \text{ if $b<a$,}\\
                 q\to\infty & \text{ if $b=a$ and $c>0$.}
   \end{cases}
\]
%
\end{theorem}

This framework only applies to patterns with a
high degree of regularity, as captured by the condition (\ref{eq:10}).
To generalise Theorem~\ref{thm:general}, it may be necessary to adapt
the concepts of balanced and strictly balanced graphs (discussed
in the introduction) to consider the densest substructures in a given pattern,
in some sense. This is left for future work, see Section~\ref{s:future}.

We will illustrate Theorem~\ref{thm:general} by applying it to 
some well-studied patterns described below.
The parameter which tends to infinity is specified for each pattern.

\begin{definition}\label{def:def}
\emph{
\begin{itemize}
\item For fixed $k\geq 3$, a  \emph{non-trivial $k$-term arithmetic progression} consists of $k$ distinct vectors 
$$
x, x+v,  \ldots, x+(k-1)v
$$ 
of $\F$.   Asymptotics are as $q+n$ tends to infinity.
\item A \emph{sum} consists of vectors $x,y,z$ (not necessarily distinct) 
such that $x+y=z$. 
\item A \emph{non-trivial parallelogram} consists of four distinct vectors $x_{1}, x_{2}, x_{3}, x_{4}$ of $\F$
such that 
\[  x_{\sigma(1)}-x_{\sigma(2)}=x_{\sigma(4)}-x_{\sigma(3)} 
\]
 for some permutation $\sigma$. Asymptotics are as $q+n$ tends to infinity.
 \item A \emph{non-trivial right triangle} consists of three distinct vectors  $x_{1}, x_{2}, x_{3}$ of $\F$ such that
\[  (x_{\sigma(2)}-x_{\sigma(1)})\cdot (x_{\sigma(3)}-x_{\sigma(1)})=0 \]
 for some permutation $\sigma$.
Asymptotics for the threshold result are as $q+n$ tends to infinity. 
For the Poisson limit result we take asymptotics as $q\to\infty$, with $n\geq 2$. 
(Here $n$ may be bounded or growing.)
\item 
An $m$-\emph{dimensional plane}  is a translation of an
$m$-dimensional subspace, for any $m\in \{1,2,\ldots, n-1\}$, where $n\geq 2$. Each $m$-dimensional plane contains $q^{m}$ vectors. 
Asymptotics are as $n$ tends to infinity, with $q$ and $m$ fixed.
\end{itemize}
}
\label{patterns}
\end{definition}

In the definition of $k$-APs, parallelograms and right triangles, ``non-trivial'' means that the vectors in the pattern are distinct. For ease of exposition we sometimes omit the word ``non-trivial'' when referring to these patterns.

Observe that (non-trivial) right triangles cannot be expressed as the solution of
a system of linear equations.  Similarly, $m$-dimensional planes do not have an
analogue in the setting of random subsets of integers.

Furstenberg and Katznelson~\cite{FKz} proved that there is a function
$Q(\varepsilon,q,m)$ defined for all positive integers $m$, prime powers $q$ and 
$\varepsilon>0$, such 
that if $n > Q(\varepsilon,q)$ and $S\subseteq \F$ satisfies $|S|\geq \varepsilon q^n$
then $S$ contains an $m$-dimensional affine subspace. 
One of our applications, below, is a threshold result for the appearance
of $m$-dimensional affine planes in random subgraphs of $\F$.


\begin{theorem}\label{thm:main} 
Let $\A$ be one of the collections of 
patterns defined in Definition~\ref{patterns}. 
The functions  
\begin{equation} 
t(n,q)=
  \begin{cases}
   q^{-2n/k} & \text{$k$-APs}, \\
   q^{-2n/3} & \text{sums}, \\
   q^{-3n/4} &\text{parallelograms},\\
   q^{-n+1/3} & \text{right triangles},\\
   q^{-(m+1)n/q^{m}} & \text{$m$-dimensional planes}
  \end{cases}
\end{equation}
 are threshold functions for the property ``$E$ contains an element of $\A$". 
Furthermore, if $\lambda_{\A}\to\lambda$ where $\lambda\in(0,\infty)$ then
$X \xrightarrow{\text{d}} \operatorname{Po}(\lambda)$, where $X$ is the
number of elements of $\A$ contained in the random set $E$, and $\lambda_{\A}$
is defined in \emph{(\ref{eq:lambdaA})}.
All asymptotics are taken 
with the respect to the limits described in Definition~\emph{\ref{patterns}}.
\end{theorem}


A set $A\subseteq \F$  is \emph{sum-free} if there are no elements $x,y,z\in \F$
with $x+y=z$. 
A \emph{Sidon set} in $\F$ is a subset $A\subseteq \F$ 
such that sums of elements of $A$ are unique, up to permutation of the 
summands.  That is, in a Sidon set $A$, 
if $x+y=z+w$ for $x,y,z,w\in A$ then $(x,y) = (z,w)$ or $(x,y) = (w,z)$.
The next two results follow easily from Theorem~\ref{thm:main}
(indeed, Corollary~\ref{cor:sum-free} is immediate).

\begin{corollary}
\label{cor:sum-free}
The function $t(n,q) = q^{-2n/3}$ is a coarse threshold for the
(decreasing) property that $E$ is sum-free, as $q+n\to\infty$.
\end{corollary}




\begin{corollary}
The function $t(n,q) = q^{-3n/4}$ is a coarse threshold for the
(decreasing) property that $E$ is a Sidon set, as $q+n\to\infty$.  
\end{corollary}

\begin{proof} 
Observe that $E$ is a Sidon set if and only if it contains no
non-trivial parallelograms and no non-trivial 3-APs.  
If $\delta = o(q^{-3n/4})$ then also $\delta = o(q^{-2n/3})$, 
so the probability that $E$ contains a non-trivial parallelogram or 
non-trivial 3-AP tends to zero.  That is, the probability
that $E$ is a Sidon set tends to~1 when $\delta =o(q^{-3n/4})$.  Conversely,
if $\delta = \omega(q^{-3n/4})$ then $E$ contains a non-trivial parallelogram
with probability which tends to~1, so the probability that $E$ is a Sidon set
tends to zero.
\end{proof}

We note that there are many other interesting patterns in $\F$, such as Kakeya sets~\cite{Dvir} and Furstenberg sets~\cite{EE}. 

%
%
\bigskip



Our proofs use standard arguments, namely the first and second moment method
and the method of moments.

We will show that if $\delta= o(t(n,q))$ then $\EE(X_{\A})$ tends to zero. 
The negative side of the threshold result then
follows directly from Markov's inequality (first moment method),  since
\begin{equation}
\label{eq:first}
\PP(X_{\A}\geq 1)\leq \EE(X_{\A}).
\end{equation}
The positive side of the threshold result 
is completed using the second moment method.
Finally, the proof of Theorem~\ref{thm:main} is completed by studying
the higher order moments of the random variable $X_{\A}$ 
and applying the method of moments.

 
The rest of the paper is organised as follows.  First we finish this section by discussing
some other related work.  In Section~\ref{sec:pre} we introduce some notation, 
then state conditions which are sufficient for the second moment and method of moments arguments.
%
We prove our general theorem, Theorem~\ref{thm:general}, in Section~\ref{sec:general}.
Then in Section~\ref{s:applications} we 
apply Theorem~\ref{thm:general} to prove Theorem~\ref{thm:main}
for $k$-APs, sums, right triangles, and parallelograms.
The case of $m$-dimensional planes requires special care and is treated separately in 
Section~\ref{sec:planes}. 
Section~\ref{sec:further} contains some further remarks, including applications to
extremal problems and directions for future research.

\subsection{Other related work}\label{s:related}

There are Ramsey-theoretic versions of the problem of existence
of a given substructure in a random combinatorial object.
For example, graph $\Gamma$ is $(G,k)$-Ramsey if there is a monochromatic copy of $G$
in any colouring of the edges of $\Gamma$ with $k$ colours.
R{\" o}dl and Ruci{\' n}ski~\cite{RR-AMS} proved that if $G$ is not
a forest of stars and paths then there are constants $c,C>0$ such that
\[ \lim_{n\to\infty} \Pr(G(n,p) \text{ is $(G,k)$-Ramsey})
  = \begin{cases} 0 & \text{ if $p < cn^{-1/m_2(G)}$,}\\ 
   1 & \text{ if $p>Cn^{-1/m_2(G)}$}\end{cases}
\]
where 
\[ m_2(G) = \max\left\{ \frac{|E(H)|-1}{|H|-2} \, :\, H\subseteq G\right\}.\]

Let $[n]_p$ denote the random subset of $[n]$, where each element is chosen
independent with probability $p$.  Schur's Theorem says that for every $r\geq 2$ and for $n$ sufficiently
large, in any $r$-colouring of $[n]$ there must be a monochromatic solution
to $x+y=z$.  A solution with $x\neq y$ is called a Schur triple.
Graham, R{\" o}dl and Ruci{\' n}ski~\cite{GRR}  proved that the threshold
for any 2-colouring of $[n]_p$ to have a monochromatic Schur triple is
$p = n^{-1/2}$, and gave a lower bound on the number of Schur triples
when $pn^{1/2}\to\infty$.

Given an $\ell\times k$ matrix $A=(a_{ij})$ of integers, say that $A$
is \emph{partition regular} if for any finite colouring of $\mathbb{N}$
there is a monochromatic solution to the system of linear equations 
$A\xvec = \zerovec$. We may assume that $A$ has full rank and that it
is \emph{irredundant}, which means that for each $i\neq j$ there
is a solution $\xvec = (x_1,\ldots, x_k)$ with $x_i\neq x_j$.
An irredundant matrix $A$ is \emph{density regular} if 
$\xvec=(1,1,\ldots, 1)$ is a solution of $A\xvec=\zerovec$.
R{\" o}dl and Ruci{\' n}ski~\cite{RR} proved a threshold result
for the existence of monochromatic solutions to $A\xvec = \zerovec$
in $[n]_p$, given any colouring of $[n]$ using $r$ colours,
when $A$ is an irredundant, density regular integer matrix $A$. 
The threshold is determined by a parameter $m_A$ which is a maximum
over all sets of columns of a function involving the rank of the
submatrix induced by the chosen columns. They conjectured that the
same threshold result should hold for any partition regular system
of equations (without the density-regular condition).

Similarly, many researchers have considered ``sparse random'' versions of
extremal problems.
Kohayakawa, {\L}uczak and Ruci{\' n}ski~\cite{KLR} studied 3-APs in $[n]_p$.
Given $0 < \alpha < 1$, a subset $A\subset [n]$ is called \emph{$\alpha$-3AP} 
if every subset $F\subseteq A$ with $|F|\geq \alpha|A|$ contains a three-term 
arithmetic progression.  Kohayakawa et al.~\cite{KLR} found a threshold for the 
probability that $[n]_p$ is $\alpha$-3AP, using the Szemer{\' e}di regularity lemma.

Conlon and Gowers~\cite{CG}, Schacht~\cite{S} and Friedgut, R{\" o}dl and 
Schacht~\cite{FGS} presented powerful methods for attacking sparse random versions 
of Ramsey-theoretic and extremal combinatorial questions. 
Several conjectured results were established in these works.
In particular, Conlon and Gowers~\cite{CG} and independently, Schacht~\cite{S}
determined the threshold for Szemer{\' e}di's theorem in random subsets
of the integers, and Friedgut et al~\cite{FGS} proved a threshold result for
Rado's theorem on solutions of partition regular systems of equations in random
subsets of integers, proving the conjecture made by R{\" o}dl and Ruci{\' n}ski~\cite{RR}.

Recently there has been a growing interest in studying finite field versions of some classical 
problems arising from Euclidean spaces. For instance, there are finite field  Besicovitch-Kakeya 
sets~\cite{Dvir,Green,Wolff}; finite field Erd\H{o}s--Falconer distance problem~\cite{IosevichRudnev} 
and finite fields version of restriction problems of Fourier analysis~\cite{M-T}.  
This was an additional motivation for our study.
From another point of view, $\Omega(\F,\delta)$ can be considered as a discrete version of 
fractal percolation (Mandelbrot percolation), see \cite[Chapter 15]{Falconer}.  
For some applications of $\Omega(\F, \delta)$ to deterministic problems, we refer 
to~\cite[Theorem 5.2]{Babai} and~\cite{Chens} for finding subsets of $\F$ with small Fourier 
coefficients, and \cite[Theorems~1.5 and~1.6]{Chenp} for finding sets without exceptional 
projections.

\section{Preliminaries}\label{sec:pre}

The following standard asymptotic notation will be used,  assuming that asymptotics are taken with
respect to a parameter $N$.
Write ``a.a.s.'' to mean \emph{asymptotically almost surely},
which means that the property holds with probability which tends to $1$. 
We write $f=O(g)$ if there is a positive constant $C$ such that $|f(N)|\leq C|g(N)|$ for all $N$ sufficiently large,
and write $f=\Omega(g)$ if $g=O(f)$. 
If $f=O(g)$ and $f=\Omega(g)$ then $f=\Theta(g)$.
We write $f = o(g)$ if $f(N)/g(N)\rightarrow 0$, and write 
$f = \omega(g)$ if $g=o(f)$. Finally, $f\sim g$ means that $f = g(1+o(1))$.

We also use $f=O_{L}(g)$ to mean that $|f(N)|\leq C|g(N)|$ for all $N$ sufficiently large, where $L$ is a list of parameters and $C$ is a positive constant which depends on $L$.  

The parameter $\lambda$ always denotes a fixed positive real number.

Recall that  $\A=\A_{a}$ is a family of subsets of $\F$ such that every element of $\A$ contains $a$ points.  Our asymptotics are with respect to
some function $N=N(q,n)$ which tends to infinity.
In the calculations for the second moment, we consider intersections of pairs of elements of $\A$.
For $k=0,1,\ldots, a$, define   
\begin{equation}\label{eq:ik}
I_{k}=I_{\A, k}=\{(T,T'): T, T' \in \A, \, |T \cap T'|=k\}.
\end{equation}
Let $Y=Y_{\A}$ be the random variable which counts pairs $(T,T')$ of distinct elements of $\A$
with non-empty intersection 
such that both $T$, $T'$ are contained in the random set $E$. That is,
\begin{align}\label{eq:y}
Y=Y_{\A}=\sum_{\substack{T,T'\in \A\\
1\leq |T\cap T'|\leq a-1}} {\bf 1}_{E}(T\cup T')
 &=\sum_{k=1}^{a-1}\sum_{(T,T')\in I_{k}}{\bf 1}_{E}(T\cup T').
\end{align}

Recall the definition of $X$ from (\ref{eq:x}).

\begin{lemma}[Second moment] \label{lem:second}
Suppose that $\EE(X)=|\A| \delta^{a}\rightarrow \infty$ with respect to some $N=N(q, n)\rightarrow \infty$, and that 
\begin{itemize}
\item[\emph{(C1)}] $|I_{0}|\sim |\A|^{2}$;
\item[\emph{(C2)}] $\EE(Y) = o(\EE(X)^{2})$. 
\end{itemize}
Then a.a.s. $E$ contains some element of $\A$. 
\end{lemma}

\begin{proof}
Observe that $|I_a| = |\A|$, and hence that $|I_a|\delta^a = \EE(X) = o(\EE(X)^2)$.  Therefore, using (C1) and (C2) we have
\begin{align*}
\EE(X^{2})&=\EE \Big(\sum_{(T, T') \in \A\times\A}{\bf 1}_{E}(T) {\bf 1}_{E}(T') \Big)\\
&=|I_{0}|\,\delta^{2a}+|I_{a}|\,\delta^{a}+\EE(Y)\\
&\sim \EE(X)^{2}.
\end{align*}
Hence, by the Paley-Zygmund inequality,
\[
\PP(X>0)\geq \frac{\EE(X)^{2}}{\EE(X^{2})}\rightarrow 1,
\]
completing the proof.  
\end{proof}

Now suppose that $\EE(X)\rightarrow \lambda$ where $\lambda \in (0, \infty)$. 
We consider the factorial moment $\EE((X)_{r})$ of the random variable $X$ for positive integers $r$. Here 
\[
(X)_{r}:=X(X-1)\ldots(X-r+1).
\]
We will show that $\EE((X)_{r})\rightarrow \lambda^{r}$ for all $r\in \NN$. 
From this, the method of moments implies that $X$ converges in distribution to a
Poisson distribution with parameter $\lambda$.

\begin{lemma}{\rm(Bollobas \cite[Theorem 1.22]{Bo})}
Let $\mu$ be a positive real constant.
If for all  $r\in\mathbb{N}$ the factorial moment $\EE((X)_r)\rightarrow \mu^r$ with respect to some $N=N(q, n)\rightarrow \infty$, then $X$ converges in
distribution to a Poisson distribution with mean $\mu$.
\label{method-moments}
\end{lemma}

Recall the definition of $I_k$ from (\ref{eq:ik}) and define 
\[ I_{\geq 1}=\bigcup_{k=1}^{a}I_{k}. \]
The following lemma can be used to check that the conditions
of Lemma~\ref{method-moments} hold.

\begin{lemma}[Higher moments] \label{lem:high}
Assume that $|\A|\rightarrow\infty$ with respect to some $N=N(q, n)\rightarrow \infty$ and let $X=X_{\A}$ be as defined in (\ref{eq:x}). 
Suppose that $\EE(X)\rightarrow \lambda$ where $\lambda$ is a positive real number, 
and that the following conditions hold:
\begin{itemize}
\item[\emph{(C1)}] $| I_{0}|\sim|\A|^{2}$; 
\item[\emph{(C2)}] $\EE(Y)=o(\EE(X)^{2})$; 
\item[\emph{(C3)}] $\EE(X^{r})=O_{r}(1)$ for all fixed integers $r\geq 2$.
\end{itemize}
Then $\EE((X)_r)\rightarrow \lambda^r$ for all $r\in\mathbb{N}$. 
\end{lemma}

\begin{proof} For any integer $r\geq 2$, define  
\[
(\A)_{r}:=\{(T_{1},\ldots, T_{r})\in \A^{r} \, : \, T_1,\ldots, T_r \text{ are pairwise distinct} \}.
\]
Let
\[
\Gamma_{0}^{(r)}=\{(T_{1},\ldots, T_{r})\in (\A)_{r} \, : \, T_1,\ldots, T_r \text{ are pairwise disjoint} \},
\]
and write
$\Gamma_{\geq 1}^{(r)}=(\A)_{r}\backslash \Gamma_{0}^{(r)}$.  Observe that
\[
|\Gamma_{\geq 1}^{(r)}|\le {r \choose 2}\, |I_{\geq1}|\, |\A|^{r-2}.
\]
 Condition (C1) implies that $|I_{\geq1}|=o(|\A|^{2})$, and hence 
\[
|\Gamma_{\geq 1}^{(r)}|=o(|\A|^{r}).
\] 
It follows that
\begin{equation}\label{eq:Gammazero}
|\Gamma_{0}^{(r)}|\sim |(\A)_{r}|.
\end{equation}
Write the random variable $(X)_{r}$ as  
\[
(X)_{r} =\sum_{\Tsequence \in (\A)_{r}} {\bf 1}_{E} (\Tsequence)
= \sum_{\Tsequence \in \Gamma_{0}^{(r)} } {\bf 1}_{E} (\Tsequence) +\sum_{\Tsequence \in \Gamma_{\geq 1}^{(r)} } {\bf 1}_{E} (\Tsequence).
\]
Applying the estimate \eqref{eq:Gammazero} and the assumption that $|\A|\rightarrow \infty$,  we obtain 
\begin{equation}\label{eq:expectationGamma}
\EE\Big(\sum_{\Tsequence \in \Gamma_{0}^{(r)} } {\bf 1}_{E} (\Tsequence)\Big)
 =|\Gamma_{0}^{(r)}|\,\delta^{ar}\sim |(\A)_r|\,\delta^{ar} \sim  \lambda^{r}.
\end{equation}
It remains to prove that $\EE\Big(\sum_{\Tsequence \in \Gamma_{\geq 1}^{(r)} } {\bf 1}_{E} (\Tsequence)\Big) = o(1)$. 

Recall that $Y$, defined in (\ref{eq:y}), counts pairs of distinct elements of $\mathcal{A}$ with non-empty intersection which are contained in $E$.
If $Y=0$ with probability one then $\sum_{t\in \Gamma_{\geq 1}^{(r)} } {\bf 1}_{E} (\Tsequence)=0$ with probability one. Hence the result holds in this trivial case.  
For the remainder of the proof, suppose that $Y>0$ with positive probability. 
Then $\EE(Y)>0$, so we may define 
\[
M=M_{Y,r}:=\EE(Y)^{-\frac{1}{2r}}.
\] 
Write
\begin{align}
\sum_{\Tsequence \in \Gamma_{\geq 1}^{(r)} } {\bf 1}_{E} (\Tsequence)
 &= {\bf 1}_{(X< M)}\sum_{\Tsequence \in \Gamma_{\geq 1}^{(r)} } {\bf 1}_{E} (\Tsequence)
 \,\, + \,\, {\bf 1}_{(X\geq M)}\sum_{\Tsequence \in \Gamma_{\geq 1}^{(r)} } {\bf 1}_{E} (\Tsequence). 
\label{eq:bad}
\end{align}
Note that for any 
$
t=(T_1, \ldots, T_r)\in \Gamma_{\geq 1}^{(r)} 
$
there exists $1\le i <  j\le r$ with $|T_i\cap T_j|\ge 1$. 
Such pairs which are contained in $E$ are counted by $Y$, and there
are $\binom{r}{2}$ choices for $(i,j)$. 
Furthermore, when $X < M$ then there are at most $M^{r-2}$ choices
for the other entries of $t$.
Therefore, using the union bound,
\[
{\bf 1}_{(X< M)}\sum_{\Tsequence \in \Gamma_{\geq 1}^{(r)} } {\bf 1}_{E} (\Tsequence)\leq \binom{r}{2}\, Y\, M^{r-2}.
\]
By definition of $M$ and by assumption (C2), we have $\EE(Y)=o(\EE(X)^{2})\rightarrow0$ and hence
\begin{equation}
\label{eq:term1}
\EE\bigg({\bf 1}_{(X<M)}\, \sum_{t\in\Gamma^{(r)}_{\geq 1}} {\bf 1}_E(t)\bigg) \leq
 \binom{r}{2}\, \EE(YM^{r-2})
  = \binom{r}{2}\, \EE(Y)^{\frac{r+2}{2r}} \rightarrow 0.
\end{equation}
So the first summand in (\ref{eq:bad}) is vanishing.

For the second summand in (\ref{eq:bad}), observe that
$
\sum_{\Tsequence \in \Gamma_{\geq 1}^{(r)} } {\bf 1}_{E} (\Tsequence) \le X^{r}.
$
Using this and the Cauchy--Schwarz inequality (in the form $|\EE(ZW)|^2\leq \EE(Z^2)\EE(W^2)$
for random variables $Z$, $W$), we obtain
\[
\EE\Big({\bf 1}_{(X\geq M)}\sum_{\Tsequence \in \Gamma_{\geq 1}^{(r)} } {\bf 1}_{E} (\Tsequence)\Big) \leq \EE\Big({\bf 1}_{(X\geq M)}\, X^r\Big)
 \leq \EE ({\bf 1}_{(X\geq M)})^{1/2}\,\, \EE(X^{2r})^{1/2}.
\]
Now assumption (C3) says that $\EE(X^{2r})= O_r(1)$, and 
\[
 \EE\big({\bf 1}_{(X\geq M)}\big) = \PP(X\geq M) \leq \frac{\EE(X)}{M} =
  \EE(X) \EE(Y)^{\frac{1}{2r}} \to 0
\]
using Markov's inequality, the definition of $M$ and the fact that $\EE(Y)\to 0$,
which follows from assumption (C2).  Hence the second term of (\ref{eq:bad}) satisfies
\begin{equation}
\label{eq:term2}
\EE\bigg({\bf 1}_{(X\geq M)}\, \sum_{t\in\Gamma^{(r)}_{\geq 1}} {\bf 1}_E(t)\bigg) 
= O_r(1)\, \Big(\EE(X)\EE(Y)^{\frac{1}{2r}}\Big)^{1/2} \to 0.
\end{equation}

Combining 
\eqref{eq:term1} and  \eqref{eq:term2} shows that \eqref{eq:bad} is vanishing, 
and together with \eqref{eq:expectationGamma}, this completes the proof.
\end{proof}

\section{Proof of the general theorem}\label{sec:general}


In this section we prove Theorem~\ref{thm:general}.  Recall that every element of
$\A$ contains $a$ vectors, and that $a$ is a fixed positive integer.

\begin{proof}[Proof of Theorem~\ref{thm:general}]\ Since $\A\subseteq \mathbb{F}_q^n$ it follows immediately that $b\leq a$.
Asymptotics are as $q+n\to\infty$ unless otherwise specified.

By definition of $t(n,q)$, if $\delta =o(t(n,q))$ then
\[ \EE(X) = |\A|\, \delta^a = |\A|\, o( q^{c-bn} ) = o(1).\]
This establishes the negative side of the threshold result.

Next we estimate the cardinality of the set $I_{k}$ defined in~\eqref{eq:ik}. 
For any $(T,T')\in I_{k}$ with $ k\in\{ 1, \ldots, b-1\}$, the sets $T$ and $T'$ intersect in 
$k$ points of $\F$. As there are $O(q^{kn})$ choices for this intersection, it follows that
\[
|I_{k}|=O(q^{kn})\,O(q^{2(b-k)n-2c})=|\A|^{2}\,O(q^{-kn}).
\]
If $k\in \{b,\ldots, a-1\}$ then $k$ points determine $O(1)$ elements of $\A$,
so $|I_k|\leq |\A|{a \choose k}\, O(1)^2$. Hence,
as $a$ is a fixed constant,
\begin{equation}\label{eq:3ik}
|I_{k}|=
  \begin{cases}
   |\A|^{2}\,O(q^{-kn}) & \text{if $1\leq k\leq b-1$}, \\
   |\A|^{2}\,O(q^{-bn+c}) &\text{otherwise}.
  \end{cases}
\end{equation}
It follows that 
\[ |I_{0}|=|A|^{2} - |A| - \sum_{k=1}^{a-1} |I_k| = |A|^2\left(1 - \,O(q^{-n+c})\right)\sim |A|^{2},\]
and condition (C1) holds.

Recall the random variables $X$ and $Y$ defined in \eqref{eq:x} and~\eqref{eq:y}, respectively.
Observe that
\begin{align*}
\EE(Y)&=\sum_{k=1}^{a-1}|I_k|\,\delta^{2a-k}\\
&=|A|^{2}\delta^{2a}\,\left(\sum_{k=1}^{b-1 }q^{-nk}\delta^{-k}+\sum_{k=b}^{a-1}q^{-bn+c}\delta^{-k}\right).
\end{align*}
Note that if $a=b$ then the summation from $b$ to $a-1$ vanishes, leaving only
the first summation.
Next, $\EE(X)=\Theta(q^{bn-c})\delta^{a}$, and hence
\begin{equation}
\label{expression}
q^{n}\, \delta=\Theta(q^{n(1-\frac{b}{a})+\frac{c}{a}}\, \EE(X)^{\frac{1}{a}}).
\end{equation}
The expression (\ref{expression}) tends to infinity if $\EE(X)\to\infty$ (whether $q$ is fixed or
tends to infinity).  If $\EE(X)\to\lambda$ for some constant $\lambda$, then the
expression in (\ref{expression}) tends to infinity whenever $b < a$, or when
$b=a, c>0$ and $q\to\infty$. 
Furthermore, for $k=b,\ldots, a$, we claim that
\[
q^{bn-c}\, \delta^{k}=\Theta(\EE(X))\, \delta^{k-a} \rightarrow \infty
\]
when $\EE(X)\rightarrow \infty$ or when $\EE(X)\rightarrow \lambda$ for some fixed positive 
$\lambda$. This is clear when $\EE(X)\to\infty$, as $k < a$;
on the other hand, if $\EE(X)\to\lambda$ for some constant $\lambda$
then we must have $\delta=o(1)$, so $\EE(X) \delta^{k-a} = \Theta(\delta^{k-a})\to\infty$.

Using the assumptions of Theorem~\ref{thm:general}, we conclude that
$\EE(Y)=o(\EE(X)^{2})$ if one of the following conditions holds:
\begin{itemize}
\item $\EE(X)\to\infty$, 
\item $b< a$ and $\EE(X)\to\lambda$ for some constant $\lambda$, and $q+n\to\infty$,
\item $b=a$, $c>0$ and $\EE(X)\to\lambda$ for some constant $\lambda$, and $q\to\infty$.
\end{itemize}
So (C2) holds when one of the above conditions holds, under the assumptions of Theorem~\ref{thm:general}. 

 %
\bigskip

If $\delta = \omega(t(n, q))$ then $\EE(X)\rightarrow \infty$, and the above arguments 
show that (C1) and (C2) hold as $q+n\to\infty$.
By Lemma~\ref{lem:second}, we conclude that a.a.s.\ $E$ contains some element of $\A$.
This completes the proof that $t(n,q)$ is a coarse threshold function for $\A$.

\bigskip

For the remainder of the proof we assume that $\EE(X)\rightarrow \lambda$ for some fixed 
positive $\lambda$. 
If $b=a$ then we additionally assume that $q\to\infty$.
We will establish (C3) by proving that $\EE(X^{r})=O_{r}(1)$, by induction on $r$. The case $r=1$ holds  by assumption. Now we suppose that $\EE(X^{r})\leq C_{r}$ holds for some $r\geq 1$. Observe that  
\begin{equation} \label{eq:Xr}
\begin{aligned}
\EE(X^r)&=\EE\left(\sum_{(T_{i_{1}}, \ldots, T_{i_{r}} )\in \A^{r}} {\bf 1}_{E}(T_{i_1})\ldots {\bf 1}_{E}(T_{i_{r}}) \right)\\
&= \sum_{(T_{i_{1}}, \ldots, T_{i_{r}} )\in \A^{r}} \PP\left(\cup_{\ell=1}^{r}T_{i_{\ell}}  \subset E\right).
\end{aligned}
\end{equation}
Moreover, by the law of total probability,
\begin{align*}
\EE(X^{r+1})
 &=\sum_{(T_{i_{1}}, \ldots, T_{i_{r+1}} )\in \A^{r+1}} \PP\left(\cup_{\ell=1}^{r+1}T_{i_{\ell}}  \subset E \right)\\
&=\sum_{(T_{i_{1}}, \ldots, T_{i_{r}} )\in \A^{r}} \,\,\, \sum_{T_{i_{r+1}}\in \A} \PP\left(T_{i_{r+1}}\subset E \mid \cup_{\ell=1}^{r}T_{i_{\ell}}  \subset E \right) \, \PP\left(\cup_{\ell=1}^{r}T_{i_{\ell}}  \subset E\right).
\end{align*}
Together with  \eqref{eq:Xr}, we conclude that it is sufficient to prove that there exists positive constant $\widehat{C}_r$ such that for any 
$T_{1}, \ldots, T_{r} \in \A$,
\begin{equation} \label{eq:induction}
\sum_{T\in \A}\PP(T\subset E \mid \cup_{\ell=1}^{r}T_{\ell}\subset E)\le \widehat{C}_r.
\end{equation}
Indeed suppose that \eqref{eq:induction} is true, then  the above argument implies  
$$
\EE(X^{r+1})\le \widehat{C}_r\, C_r =O_r(1),
$$
and we may take $C_{r+1} = \widehat{C}_r \, C_r$.

We now turn to the proof of \eqref{eq:induction}. 
Let $B=\bigcup_{i=1}^{r}T_{i}$. For $0\leq k\leq a $, define 
\[
J_{k}=\{T \in \A \, : \, |T \cap B|=k\}.
\]
It follows from \eqref{eq:3ik} that
\begin{equation*}
|J_{k}|=
  \begin{cases}
   O_{r}(q^{(b-k)n-c}) & \text{if $0\leq k\leq b-1$}, \\
  O_{r}(1) &\text{otherwise}.
  \end{cases}
\end{equation*}
We know that $q^n\delta\to \infty$, as proved below (\ref{expression}),
requiring the additional assumption that $q\to\infty$ when $b=a$.
Therefore
\begin{align*}
\sum_{T\in \A}\PP(T \subset E \mid \cup_{\ell=1}^{r}T_{\ell}\subset E)&=\sum^{a}_{k=0}|J_{k}|\,\delta^{a-k}\\
&=\sum^{b-1}_{k=0}\lambda\, O(q^{-kn}\,\delta^{-k})+\sum^{a}_{k=b}O(1)\,\delta^{a-k}\\
&=O_{\lambda,r}(1).
\end{align*}
This shows that (C3) holds, and hence by Lemma~\ref{lem:high} we conclude that
 $X \xrightarrow{\text{d}} \operatorname{Po}(\lambda)$, completing the proof.
\end{proof}

What can we say about the distribution of $X$ in the case that $\EE(X)\to\lambda$
for some positive constant $\lambda$, but $b=a$ and $q$ is fixed?
From (\ref{expression}) we see that 
\[ q^n\delta = \Theta(q^{c/a}\, \EE(X)^{1/a}) = \Theta(1),\]
which implies that $\EE(Y) = \Theta( \EE(X)^2)$.  Thus condition (C2) fails,
and the argument of Lemma~\ref{lem:high} breaks down, as it is no longer true
that most pairs of patterns in $E$ are disjoint. It is possible that in this
case $X$ has a normal distribution.  We leave this as an open question.

\section{Applications}\label{s:applications}

The desired results for almost all of the patterns described in Definition~\ref{def:def}
follow immediately from Theorem~\ref{thm:general}.
The argument for $m$-dimensional subspaces is slightly more delicate and is given 
in Section~\ref{sec:planes} below.

\begin{corollary}\label{cor:applications}
Theorem \ref{thm:main} holds for $k$-APs, sums, parallelograms and right triangles.
\end{corollary}

\begin{proof}
It suffices to prove that the conditions of Theorem~\ref{thm:general} hold with the parameters 
given in Table~\ref{t:params}, 
noting that the asymptotic assumptions of Theorem~\ref{thm:general} match those specified in 
Definition~\ref{def:def}. 
\begin{table}[ht!]
\begin{center}
\renewcommand{\arraystretch}{1.2}
 \begin{tabular}{|c |c |c| c|} 
\hline
$\A_{a}$ & $a$ & $b$ & $c$ \\  
\hline
 $k$-APs & $k$ & 2 &  0\\ 
\hline
sums & 3 & 2 &  0\\ 
\hline
parallelograms & 4 & 3  & 0 \\
 \hline
right triangles & 3 &  3 & 1  \\
\hline
\end{tabular}
\caption{Parameter values for $k$-APs, sums, parallelograms and right triangles}
\label{t:params}
\end{center}
\end{table}

We give the proof for right triangles only: the proof for the other patterns follow similarly. 
Let $\A$  be the set of all right triangles in $\FF_{q}^n$.
Clearly $a=3$ as each right triangle contains $3$ points. Each right triangle can be chosen as follows: 
we first choose two distinct vectors $x, y\in \FF_{q}^{n}$, in $\Theta(q^{2n})$ different ways, then we choose a 
vector $z\in \FF_{q}^{n}$ such that 
\begin{equation}\label{eq:orthogonal}
(z-x)\cdot(y-x)=0 \quad \text{ or } \quad (z-y) \cdot (x-y)=0.
\end{equation}
Since $\{\xi\in \FF_{q}^{n}: \xi \cdot (x-y)=0\}$ is a $(n-1)$-dimensional subspace, there are $\Theta(q^{n-1})$ choices of $z$ such that  \eqref{eq:orthogonal} holds. 
Thus $|\A|=O(q^{3n-1})$. For a lower bound, observe that each right triangle in $\FF_{q}^{n}$ was chosen at most $O(1)$ times  using this process. Furthermore, the above process may also produce triples $(x, y, z)$ with $z=x$ or $z=y$, and 
these are not right triangles. Therefore
\[
\Theta(q^{3n-1})\leq O(1)\, |\A|\, +O(q^{2n}).
\]
Since $n\geq 2$ it follows that $|\A| = \Theta (q^{3n-1})$, 
so we take $b=3$ and $c=1$ as in Table~\ref{t:params}. 
Furthermore the above arguments also implies that the condition \eqref{eq:10} holds for right triangles.
\end{proof}

\subsection{Affine planes}\label{sec:planes}
Let $G(n,m)$ be the collection of all $m$-dimensional linear subspaces of $\F$, and  $\A=A(n,m)$ be the family of all $m$-dimensional planes. It is not hard to obtain (see \cite[Theorem 6.3]{Cameron} or the proof of Lemma~\ref{lem:intersection}(i), below),
\begin{equation*}
|G(n,m)|=\frac{(q^{n}-1)(q^{n}-q)\ldots (q^{n}-q^{m-1})}{(q^{m}-1)(q^{m}-q)\ldots (q^{m}-q^{m-1})},
\end{equation*}  
and hence
\begin{equation}\label{eq:G}
|G(n,m)|=\Theta_{q, m}(q^{mn}).
\end{equation}
Since every element of $A(n,m)$ is obtained by translating some $m$-dimensional plane, and  $q^{n-m}$ elements of $A(n, m)$ are obtained by translating a given $m$-dimensional plane, we obtain
\begin{equation}\label{eq:A(n,m)}
|A(n,m)|=|G(n,m)|\, q^{n-m}=\Theta_{q,m}(q^{(m+1)n}).
\end{equation}
By definition of  $t(n, q)=q^{-\frac{(m+1)n}{q^{m}}}$, if $\delta=o(t(n, q))$ then 
\[
\EE(X)=|A(n,m)|\, \delta^{q^{m}}=o(1).
\]
This establishes the negative side of the threshold result. 

Note that for any $T, T' \in A(n,m)$, the intersection  $T\cap T'$ is ether empty or a  $k$-dimensional plane for some $k=0, \ldots, m$. Recall the definition of $I_{q^{k}}$ from \eqref{eq:ik}.

\begin{lemma}\label{lem:intersection} Using above notations, we obtain the following estimates.
\begin{itemize}
\item[\emph{(i)}]  Let $V\in A(n,k)$. Then for $m\geq k,$
\[
|\{W\in A(n,m): V\subset W\}|=|G(n-k,m-k)|=\Theta (q^{(m-k)n}).
\]
\item[\emph{(ii)}]  For $k=0, \ldots, m$ we have
\[
|I_{q^{k}}|\leq |A(n,k)|| G(n-k,m-k)|^{2}=| A(n,m)|^{2} \, O(q^{-(k+1)n}).
\]
It follows that 
\[
| I_{\geq 1}|=| A(n,m)|^{2}\, O(q^{-n}),
\]
and hence 
\[
|I_{0}|= |A(n,m)|^{2}\, (1+O(q^{-n})).
\]
\item[\emph{(iii)}] Recall that $\EE(X)=| A(n,m)|\delta^{q^{m}}=:\lambda_{\A}$. If $\lambda_{\A}\rightarrow \infty$ or $\lambda_{\A}\rightarrow \lambda$ for some positive $\lambda$, then for any $ k=0, \ldots, m-1$, 
\begin{equation}
q^{-n(k+1)}\delta^{-q^{k}}=o(1).
\end{equation}
\end{itemize}
\vspace*{-\baselineskip}
\end{lemma}
\begin{proof}

For (i), we  assume that $m>k$ and $V\in G(n,k)$. Observe that to obtain a $m$-dimensional subspace which contains $V$, it is sufficient to choose a set $S=\{u_{1}, \ldots, u_{m-k}\}$ of $\F$ such that $S\cup V$ spans an $m$-dimensional subspace. There are $q^{n}-q^{k}$  choices for $u_{1}$ (all except the  vectors from the subspace $V$).  To choose $u_{2}$, we have $q^{n}-q^{k+1}$  choices to ensure that $u_{2}$ is independent of $\{u_{1}\}\cup V$, and so on. In the end there are 
\[
(q^{n}-q^{k})(q^{n}-q^{k+1})\ldots (q^{n}-q^{m-1})
\]  ways to choose $S$. Note that for any $m$-dimensional subspace which contains $V$, there are  
\[
(q^{m}-q^{k})(q^{m}-q^{k+1})\ldots (q^{m}-q^{m-1})
\] choices of $S$ which generate (or span) the same subspace. It follows that the  number of $m$-dimensional subspaces which contain $V$ is 
\[
\frac{(q^{n}-q^{k})(q^{n}-q^{k+1})\ldots (q^{n}-q^{m-1})}{(q^{m}-q^{k})(q^{m}-q^{k+1})\ldots (q^{m}-q^{m-1})}=|G(n-k,m-k)|,
\] 
and  applying \eqref{eq:G} completing the proof of (i).

To prove  (ii), note that $T$ intersects $T'$ at some $k$-dimensional plane for any $(T, T')\in I_{q^{k}}$. For every $k$-plane there are  $|G(n-k,m-k)|$  $m$-planes which contain it, and hence by (i) above we have
\[
|I_{q^{k}}|\, \leq |A(n,k)|| G(n-k,m-k)|^{2} 
= | A(n,m)|^{2}\, O(q^{-n(k+1)}).
\]
To establish (iii), since $\lambda_{\A}=|A(n,m)|\delta^{q^{m}}$, by \eqref{eq:A(n,m)} we have
\[
\delta^{q^{m}}=\lambda_{\A}\, q^{-n(m+1)}\, \Theta(1).
\]
Then
\begin{align*}
q^{n(k+1)}\, \delta^{q^{k}}&=q^{n(k+1)}\,\lambda_{\A}^{q^{k-m}}\, q^{-n(m+1)\,q^{k-m}}\, \Theta(1)\\
&=\lambda_{\A}^{q^{k-m}}q^{n(k+1-(m+1)q^{k-m})}\,\Theta(1).
\end{align*}
Observe that for any $x\geq 0$ and $q\geq 3$, 
\[
\frac{x+1}{q^{x}}>\frac{x+2}{q^{x+1}}.
\]
It follows that $k+1>(m+1)q^{k-m} $ for any $k=0, 1, \ldots, m-1$. Define  
\[
\alpha=\alpha(q,m)=\min\{k+1-(m+1)q^{k-m}\,: k\in \NN, \, k=0, \ldots, m-1\}.
\]
Thus $\alpha >0$. It follows that for any $ k=0, \ldots, m-1$,
\[
q^{n(k+1)}\,\delta^{q^{k}}=\lambda_{\A}^{q^{k-m}}\,q^{n\alpha}\,\Omega_{q,m}(1),
\]
completing the proof. 
\end{proof}

We immediately have the following consequence. 

\begin{lemma}\label{lem:yy}
Suppose that $\EE(X)\rightarrow \infty$ or $\EE(X)\rightarrow \lambda$  for some fixed $\lambda >0$. Then $\EE(Y)=o(\EE(X)^{2})$, so \emph{(C2)} holds.
\end{lemma}
\begin{proof}
Observe that, using Lemma \ref{lem:intersection}(ii),  
\begin{equation*}
\EE(Y)=\sum_{k=1}^{m-1}|\,I_{q^{k}}|\,\delta^{2q^{m}-q^{k}}=\EE(X)^2 \sum_{k=1}^{m-1}O(q^{-n(k+1)}\,\delta^{-q^{k}}).
\end{equation*}
Lemma~\ref{lem:intersection}(iii) implies that 
\[
\sum_{k=1}^{m-1}O(q^{-n(k+1)}\,\delta^{-q^{k}})=o(1),
\] 
as required.
\end{proof}

Recall that $t(n, q)=q^{-\frac{(m+1)n}{q^{m}}}$. If $\delta=\omega(t(n,q))$ then $\EE(X)=|A(n, m)| \delta^{q^{m}}\rightarrow \infty$. Lemma \ref{lem:intersection}(ii) and Lemma \ref{lem:yy} implies that Lemma \ref{lem:second} holds for $m$-dimensional planes. 
This establishes the coarse threshold result of Theorem~\ref{thm:main} 
for $m$-dimensional planes.

For later use, we state the following easy fact as a lemma.  It follows directly from Lemma \ref{lem:intersection}(i).

\begin{lemma}\label{lem:cover}
Given $F=\{x_{1}, \ldots, x_{t}\}\subseteq \F$, let $d(F)$ be the smallest integer $k$ such that $F\subseteq V$ for some $V\in A(n, k)$. Then for $m\geq d(F)$,
\[
|\{T\in A(n,m): F\subseteq T\}|=|G(n-d(F),m-d(F))|=\Theta_{q, m}(q^{n(m-d(F))}).
\]
\end{lemma}

Now we show that if $\EE(X)\rightarrow \lambda$ for some fixed positive $\lambda$ then for  any $r\in \NN$, $\EE(X^{r})=O_{r}(1)$. Applying the same argument as in the proof of Theorem~\ref{thm:general}, it is sufficient to prove the following result.

\begin{lemma}\label{lem:condition}
Assume that $\EE(X)\rightarrow \lambda$ for some fixed positive $\lambda$. Let $T_{1}, \ldots, T_{r} \in A(n,m)$. Then
\begin{equation*}
\sum_{T\in A(n,m)}\PP(T\subseteq E \mid \cup_{i=1}^{r}T_{i}\subseteq E)=O_{q,m,\lambda,r}(1).
\end{equation*}
\end{lemma}
\begin{proof}
Let $B=\cup_{i=1}^{r}T_{i}$. For $k=0, 1, \ldots, q^{m}$, define 
\[
J_{k}=\{T \in A(n,m): |T \cap B|=k\}.
\]
We will split  the sum up as follows:
\begin{align}
& \sum_{T\in A(n,m)}\PP(T\subseteq E  \mid \cup_{i=1}^{r}T_{i}\subseteq E) \nonumber \\
&\qquad =\sum_{k=0}^{q^{m}}|J_{k}|\, \delta^{a-k} \nonumber \\
&\qquad =|J_{0}|\, \delta^{a}  +|J_{1}|\, \delta^{a-1}+ \sum_{j=1}^{m-1} \sum_{q^{j-1}<
 k\leq q^{j}} |J_{k}|\, \delta^{a-q^{j}}+\sum_{q^{m-1}<k\leq q^{m}}|J_{k}| \, \delta^{a-k}.
\label{eq:split}
\end{align}
Let $a:=q^{m}$ and $\lambda=\EE(X)$.
First, $|J_{0}|\leq |A(n,m)|$ so 
\begin{equation}
\label{eq:0}
|J_{0}|\delta^{a}\leq |A(n, m)|\delta^{a}=\lambda=O(1).
\end{equation}
Next, by (\ref{eq:A(n,m)}),
\[
 |J_{1}|\leq |B| \, |G(n, m)|=O(q^{-n})\, |A(m, n)|.
\]
Applying Lemma \ref{lem:intersection}(iii) we have 
\begin{equation}
|J_{1}|\delta^{a-1}=O( q^{-n} \delta^{-1} \lambda ) =o(1).
\label{eq:1}
\end{equation}
Now let  $k\in \{2, 3, \ldots, q^{m}\}$. Then there exists $j\in \{ 1, \ldots, m\}$ such that   $q^{j-1}<k\leq q^{j}$. 
Let $F\subseteq B$ with $|F|=k$.  The smallest plane which contains $F$ has dimension at least~$j$.
Hence Lemma \ref{lem:cover} implies that 
\[
|\{T\in A(n, m): F\subseteq T\}|
   =O(q^{n(m-j)}).
\]
Since $|B|\leq r q^{m}=O(1)$, there are ${|B| \choose k}=O(1)$ different choices  for a subset $F$
consisting of
$k$ points of $B$. Therefore, by (\ref{eq:A(n,m)}),
\begin{equation*}
|J_{k}|=O(q^{n(m-j)})=O(q^{-n(j+1)})\,|A(n,m)|
\end{equation*}
and since $q, m$ are fixed, we obtain
\begin{equation}
\label{sumj}
\sum_{q^{j-1}<k\leq q^{j}} |J_{k} |= O(q^{-n(j+1)})\, |A(n,m)|.
\end{equation}
If $j \leq m-1$ then applying Lemma  \ref{lem:intersection}(iii) gives
\begin{align}
\sum_{j=1}^{m-1} \sum_{q^{j-1}<k\leq q^{j}}\, |J_{k}|\, \delta^{a-q^{j}} 
&=\sum_{j=1}^{m-1} O(q^{-n(j+1)}\, \delta^{-q^{j}})\,  |A(n,m)|\,  \delta^{a} \nonumber \\
&= o(1).  
\label{eq:thirdterm}
\end{align}
Finally, when $j=m$ we have, by (\ref{eq:A(n,m)}) and (\ref{sumj}),
\begin{equation}
\sum_{q^{m-1}<k\leq q^{m}}\, |J_{k}| \delta^{a-k} \leq
\sum_{q^{m-1}<k\leq q^{m}}\, |J_{k}| = 
O(1).
\label{eq:m}
\end{equation}
The result follows by substituting (\ref{eq:0}), (\ref{eq:1}), (\ref{eq:thirdterm}) and (\ref{eq:m}) into (\ref{eq:split}).
\end{proof}

Applying  Lemma \ref{lem:condition} and arguing  in the proof of Theorem~\ref{thm:general}, we obtain that condition (C3) from Lemma~\ref{lem:high} holds for $m$-dimensional planes. Together with Lemma \ref{lem:intersection}(ii) and Lemma \ref{lem:yy}, we  obtain Lemma~\ref{lem:high} (C1) and (C2), completing the proof of Theorem~\ref{thm:main} for $m$-dimensional planes.

\bigskip

\noindent {\bf Extensions and limitations:}\ 
It is possible that our analysis of $m$-dimensional subspaces in Theorem~\ref{thm:main}
can be extended to allow $m$ or $q$ to grow slowly with $n$.
Bounding $|\A|$ more carefully in (\ref{eq:A(n,m)}) suggests that the threshold function 
should be adjusted to $t(n,q) = q^{-\frac{(m+1)(n-m)}{q^m}}$ in that case.
This function is very close to~1  unless $n$ grows faster than $\frac{q^m}{(m+1)\log q}$.

\section{Further results and future directions}\label{sec:further}

We now give some applications to extremal problems, discuss an Erd\H{o}s--R{\' e}nyi
variant of the random model and mention some open problems.

\subsection{Applications to extremal problems}

Borrowing notation from extremal graph theory, let
$\operatorname{ex}(\F,\A)$ denote the  maximal cardinality of subsets of $\F$ which are $\A$ -free. 
That is,  
\[
\operatorname{ex}(\F,\A)=\max\{|B|: B \subseteq \F,\,\, B \text{ does not contain any element of $\A$ }\}.
\]
For the random variable $X$, Markov's inequality implies that 
\begin{equation}\label{eq:markov}
\PP(X\geq 1)\leq \EE(X)=|\A|\, \delta^{a}.
\end{equation}
Chebyshev's inequality implies that for our random subset $E$,   
\begin{equation}\label{eq:chebyshev}
\PP\big(| | E| - q^{n}\delta |\geq 
    \dfrac{1}{2}q^{n}\delta\big)\leq \frac{4q^{n}\delta(1-\delta)}{(q^{n}\delta)^{2}} 
             = \frac{4(1-\delta)}{q^n \delta}.
\end{equation}
That is, $|E|$ is concentrated around $q^n \delta$ when $q^n\delta$ is large.

Applying the estimates \eqref{eq:markov} and \eqref{eq:chebyshev} leads to the following lower bound on
$\operatorname{ex}(\F,\A_{a})$. 

\begin{lemma} Suppose that $|\A| \delta^{a}= \nfrac{1}{2}$  and $q^{n}\delta\geq 100$. Then there exists a subset 
$S\subseteq \mathbb{F}_q^n$ which is $\A$-free and satisfies $|S| =\Theta(q^{n}\delta)$. This implies that 
\[
\operatorname{ex}(\F, \A)=\Omega(q^{n}\delta)=\Omega(q^{n}|\A|^{-\frac{1}{a}}).
\]
\end{lemma}

This leads to the following lower bounds for the extremal problem for the patterns defined
in Definition~\ref{patterns}.

\begin{corollary}
As usual, let $q\geq 2$ be a prime power and let $n$ be a positive integer.
Let $\A_{1}$, $\A_2$ and $\A_{3}$ be the collection of all non-trivial $k$-APs, sums and non-trivial parallelograms in $\mathbb{F}_q^n$, respectively.
Similarly, let $\A_4$ denote the collection of all non-trivial right triangles 
in $\mathbb{F}_q^n$, where $n\geq 2$, and let 
$\A_5$ denote the collection of all 
$m$-dimensional planes in $\mathbb{F}_q^n$,
where $m,q$ are fixed and $q\geq 3$. Then 
\begin{align*}
\operatorname{ex}(\F,\A_{1})=\Omega(q^{n(1-\frac{2}{k})}), & \qquad \operatorname{ex}(\F,\A_{2})=\Omega(q^{\frac{n}{3}}), \qquad \operatorname{ex}(\F,\A_{3})=\Omega(q^{\frac{n}{4}}),\\
\operatorname{ex}(\F,\A_{4})=\Omega(q^{\frac{1}{3}}), &\qquad \operatorname{ex}(\F,\A_{5})=\Omega\left(q^{n(1- \frac{m+1}{q^{m}})}\right).
\end{align*}
\end{corollary}

Note that the extremal problem for  $3$-APs in $\mathbb{F}_{3}^{n}$ is called the cap set problem, and a much stronger lower bound is known. For the cap set problem, the best known bounds are  
\[
2.2^{n}\leq \operatorname{ex}(\mathbb{F}_{3}^{n}, \text{ 3-AP} )\leq 2.756^{n}.
\]    
See \cite{Edel} for the lower bound and \cite{EG} for the upper bound (and
further background).
To our knowledge the other lower bounds appear to be new.  

\subsection{Erd{\H o}s-R\'enyi model for finite vector spaces}

In this subsection, we consider another random model in $\F$ which is an  analogue of Erd\H{o}s-R\'enyi random graphs. 
Let $M=M_{n,q}\leq q^{n}$ be a positive integer.  Choose $E=E^{\omega}$ uniformly at random from the set of all subsets of $\F$ of cardinality $M$. Denote this probability space by $\Omega(\F, M)$.

Note that for a subset $F\subseteq \F$ with $|F|\leq M$, we have 
\begin{equation}\label{eq:ppp}
\PP(F \subseteq E)=\frac{M(M-1)\ldots(M-|F|+1)}{q^{n}(q^{n}-1)\ldots(q^{n}-|F|+1)}.
\end{equation}
It follows that if $|F|=O(1)$ and $M_{n,q}\rightarrow \infty$ then the identity \eqref{eq:ppp} becomes 
\begin{equation*}
\PP(F\subseteq E)\sim \left(\frac{M_{n,q}}{q^{n}}\right)^{|F|}.
\end{equation*}
For the model $\Omega(\F,M)$, we can obtain  similar results to Theorems \ref{thm:main} 
by taking the ratio $M_{n,q}/q^{n}$ instead of $\delta$ in our former proofs.  We omit these arguments.

\subsection{Future directions}\label{s:future}


Our results leave many open questions in the analysis of substructures in 
random subsets of $\F$.
Some possible directions for future study are listed below.

\begin{itemize}
\item  
To what extent can condition (\ref{eq:10}) in Theorem~\ref{thm:general}
be relaxed or removed?  
Recall that in random graphs, the threshold for existence of a given
subgraph $G$ is controlled by the parameter $m(G)$, which is the
density of the densest subgraph of $G$.  To remove condition (\ref{eq:10}) 
it may be necessary to
define an analogous parameter for substructures in $\F$.

\item

We have only analysed the limiting distribution of the random variable $X$,
defined in~\eqref{eq:x}, in the case that $\EE X$ tends to a constant
(that is, within the threshold scale).  
In this situation we proved that the limiting distribution is Poisson (in some cases 
we could prove this only when $q\to\infty$).  
A natural question is to ask for conditions under which the
limiting distribution of $X$ is normal (see~\cite{R} for the random graph setting).
This would involve application of the method of moments, or Stein's method.

\item
A fruitful line of research in recent years (for example~\cite{CG,KLR,RR,S}),
 seeks what Conlon and Gowers~\cite{CG} call ``sparse random'' analogues of extremal 
and Ramsey-theoretic problems.
In our setting, let $\A$ be a given pattern and recall that $E$ is a random
subset of $\F$ chosen according to $\Omega(\F,\delta)$.
One open problem is to find a threshold function 
for the property that, in every $r$-colouring of $E$, there is a monochromatic 
element of $\A$.  
Another aim is to establish, for a given pattern $\A$,
a threshold function for the probability that 
every subset $F\subseteq E$ with
$|F|\geq \alpha |E|$ must contain an element of $\A$.

\end{itemize}

\subsection*{Acknowledgements}\ The authors would like to thank the referees for their
extremely helpful comments which have greatly improved this paper.


\begin{thebibliography}{99}

\bibitem{Alon} N.~Alon and J.~Spencer, \emph{The Probabilistic Method}, 
Wiley Interscience, New York, 2000.


\bibitem{Babai} L.~Babai. Fourier transforms and equations over finite Abelian groups: 
An introduction to the method of trigonometric sums.  Lecture notes, available from
\url{http://people.cs.uchicago.edu/~laci/reu02/fourier.pdf}


\bibitem{Bi} P. Billingsley, \emph{Probability and Measure} (3rd ed.), Wiley, New York, 1995.

\bibitem{Bo1981} B.~Bollob{\' a}s, Threshold functions for small subgraphs,
\emph{Tham.\ Proc.\ Camb.\ Philos.\ Soc.} {\bf 90} (1981), 187--206.


\bibitem{Bo} B. Bollob\'as,  \emph{Random Graphs} (2nd ed.), Cambridge University Press, 2001.

\bibitem{BT} B.~Bollob{\' a}s and  A.~Thomason, Threshold functions, \emph{Combinatorica}
{\bf 7(1)} (1987), 35--38.

\bibitem{B-K-T} J. Bourgain, N. Katz and T. Tao, A sum-product estimate in finite fields and their applications , \emph{Geom.\ Func.\ Anal.} {\bf 14} (2004), 27--57.


\bibitem{Cameron} P.~Cameron, The Art of Counting. Lecture notes, available from
\url{http://cameroncounts.files.wordpress.com/2016/04/acnotes1.pdf}

\bibitem{Chens} C. Chen, Salem sets in vector spaces over finite fields, \emph{Ark.\ Mat.} {\bf 56}  (2018),  45--52.


\bibitem{Chenp} C. Chen, Projections in vector spaces over finite fields,
\emph{Ann.\ Acad.\ Sci.\ Fenn.\ Math.} 43, (2018), 171--185. 

\bibitem{CG}  D. Conlon and W. T. Gowers, Combinatorial theorems in sparse random sets, \emph{Ann.\ of Math.\ (2)} {\bf 184} (2016), 367--454. 

\bibitem{Dvir} Z. Dvir, On the size of Kakeya sets in finite fields, \emph{J.\ Amer.\ Math.\ Soc} {\bf 22} (2009), 1093--1097.


\bibitem{EE} J. Ellenberg and D. Erman. Furstenberg sets and Furstenberg schemes over finite fields. \emph{Algebra and Number Theory}, {\bf 10} (2016), 1415--1436.


\bibitem{EG} J. Ellenberg and D. Gijswijt, On large subsets of $\F$ with no three-term arithmetic progression,
\emph{Ann.\ of Math.\ (2)} {\bf 185} (2017), 339--343. 




\bibitem{ERII} P. Erd\H{o}s, A. R\'enyi,  On the evolution of random graphs,  \emph{Publications of the Mathematical Institute of the Hungarian Academy of Sciences}, {\bf 5} (1960), 17--61. 

\bibitem{Edel} Y. Edel, Extensions of generalized product caps, 
\emph{Des.\  Codes Cryptogr.} {\bf 31(1)} (2004), 5--14.

\bibitem{Falconer} K.~Falconer, \emph{Fractal Geometry: Mathematical Foundation and Applications}, Wiley, 
New York, 1990.

\bibitem{FK} A. Frieze and M. Karo\'nski, \emph{Introduction to Random Graphs}, Cambridge University Press, 
Cambridge, 2015. 

\bibitem{FGS}
E.~Friedgut, V.~R{\" o}dl and M.~Schacht, Ramsey properties of random
discrete structures, \emph{Random Structures Algorithms} {\bf 37} (2010),
407--436.

\bibitem{FKz}
H.~Furstenberg and Y.~Katznelson, An ergodic Szemer{\' e}di theorem for
IP-systems and combinatorial theory, \emph{J.\ Analyse Math.} {\bf 45} (1985),
117--168.

\bibitem{GRR}
\cg{
R.~Graham, V.~R{\" o}dl and A.~Ruci{\' n}ski, On Schur properties of random subsets
of integers, \emph{J.\ Number Theory} {\bf 61} (1996), 388--408.
}


\bibitem{Green-K} B. Green, Restriction and Kakeya phenomena, Lecture notes (2003).

\bibitem{Green} B. Green, Finite field models in additive combinatorics, in \emph{Surveys in Combinatorics 2005}, London Mathematical Society Lecture Note Series, Volume 327, pp. 1--27 (Cambridge University Press, 2005).

\bibitem{IosevichRudnev} A. Iosevich and M. Rudnev, Erd\H{o}s distance problem in vector spaces over finite fields, \emph{Trans.\ Amer.\ Math. Soc}, {\bf 
359} (2007), 6127--6142.

\bibitem{KR} M.~Karo{\' n}ski and A.~Ruci{\' n}ski, On the number of strictly
balanced subgraphs of a random graph, in \emph{Graph Theory, Proceedings {\L}ag{\' o}w
1981}, Lecture Notes in Mathematics vol.~1018, Springer, Berlin, 1983, pp.~79--83.

\bibitem{KSV} J.-H.~Kim, B.~Sudakov and V.~Vu, Small subgraphs of random regular
graphs, \emph{Discrete Matheamtics} {\bf 307} (2007), 1961--1967.

\bibitem{KLR} Y. Kohayakawa, T. {\L}uczak and V. R\"odl, Arithmetic progressions of length three in subsets of a random set, \emph{Acta Arith.} {\bf 75} (1996), 133--163.

\bibitem{KLRS}
Y.~Kohayakawa, S.\ J.~Lee, V.~R{\" o}dl and W.~Samotij, The number of Sidon sets
and the maximum size of  Sidon sets contained in a sparse random set of
integers, \emph{Random Structures Algorithms} {\bf 46} (2015), 1--25.

\bibitem{M-T} G. Mockenhaupt and T. Tao, Restriction and Kakeya phenomena for finite fields, 
\emph{Duke Math.~J.}, {\bf 121} (2004), 35--74.

\bibitem{RR-AMS} V. R\"odl and A. Ruci\'nski, Threshold functions for Ramsey properties,
\emph{J.~Amer.\ Math.\ Soc.} {\bf 8} (1993), 917--942.

\bibitem{RR} V. R\"odl and A. Ruci\'nski, Rado partition theorem for random subsets of integers, \emph{Proc.\ London Math.\ Soc.} {\bf 74} (1997), 481--502.


\bibitem{R} A. Ruci\'nski, When are small subgraphs of a random graph normally distributed? \emph{Probab.\ Theory Related Fields.} {\bf 78} (1988), 1--10.

\bibitem{RV} A.~Ruci{\' n}ski and A.~Vince, Strongly balanced graphs and random graphs,
\emph{J. Graph Theory} {\bf 10} (1986), 251--264.

\bibitem{RSZ} J. Ru\'e, C. Spiegel and A. Zumalac\'arregui, 
Threshold functions and Poisson convergence for systems of equations in random sets, 
\emph{Math.\ Z.} {\bf 288} (2018), 333--360.

\bibitem{S} M. Schacht, Extremal results for random discrete structures, \emph{Ann.\ of Math.\ (2)} {\bf 184} (2016), 333--365. 

\bibitem{SS} P. Shmerkin and V. Suomala,  Patterns in random fractals,
to appear in \emph{Amer.\ J.\ Math.}.  \arxiv{1703.09553}



\bibitem{TaoVu} T. Tao and V. Vu, \emph{Additive Combinatorics}, Cambridge Univ. Press, 2006.

\bibitem{Wolf} J. Wolf,  Finite field models in arithmetic combinatorics--ten years on, \emph{Finite Fields Appl.} {\bf 32} (2015), 233--274.

\bibitem{Wolff} T. Wolff, Recent work connected with the Kakeya problem. Prospects in mathematics
(Princeton, NJ, 1996), (1999), 129--162.

\end{thebibliography}
\end{document}